\documentclass[12pt]{amsart}
\usepackage{setspace} 
\usepackage[top=1 in,bottom=1 in, left=1 in, right= 1 in, includehead,includefoot]{geometry}
\usepackage[dvips]{graphicx} 
\usepackage[titletoc,title]{appendix}
\usepackage{amsmath,amsthm,amsfonts}
\usepackage[bottom]{footmisc}
\RequirePackage[colorlinks,citecolor=blue,urlcolor=blue,linkcolor=blue]{hyperref}
\hypersetup{
colorlinks = true,
citecolor=blue,
urlcolor=blue,
linkcolor=blue,
pdfauthor = {Alexey Kuznetsov, Yizao Wang},
pdfkeywords = {Laplace transform; Markov process; Markov semigroup; Hilbert space; Levy process; birth-and-death process},
pdftitle = {On the dual representations of Laplace transforms of Markov processes},
pdfpagemode = UseNone
}

       \def\mfA{\mathfrak A}
   \def\mfB{\mathfrak B}

    \def\re{\textnormal {Re}}
    
    \def\p{{\mathbb P}}
    \def\e{{\mathbb E}}
    \def\r{{\mathbb R}}
    \def\c{{\mathbb C}}
    	
    \def\d{{\textnormal d}}
    \def\i{{\textnormal i}}

    \def\calX{\mathcal X}
   \def\calY{\mathcal Y}
   \def\calP{\mathcal P}
   \def\calQ{\mathcal Q}
   \def\calF{\mathcal F}
   \def\calG{\mathcal G}

    \def\bfT{\boldsymbol T}
    \def\bfS{\boldsymbol S}
    
\newtheorem{theorem}{Theorem}[section]
\newtheorem{lemma}{Lemma}[section]
\newtheorem{proposition}{Proposition}[section]
\newtheorem{corollary}{Corollary}[section]
\theoremstyle{definition}

\newtheorem{assumption}{Assumption}[section]
\newtheorem{remark}{Remark}[section]

\numberwithin{equation}{section}

\title[Dual representations of Laplace transforms]{On the dual representations of Laplace transforms of Markov processes}
\author{Alexey Kuznetsov}
\address
{
Alexey Kuznetsov\\
Department of Mathematics and Statistics\\
York University, 4700 Keele Street
\\ Toronto, Ontario, M3J 1P3, Canada
}
\email{akuznets@yorku.ca}

\author{Yizao Wang}
\address
{
Yizao Wang\\
Department of Mathematical Sciences\\
University of Cincinnati\\
2815 Commons Way\\
Cincinnati, OH, 45221-0025, USA.
}
\email{yizao.wang@uc.edu}

\keywords{Laplace transform; Markov process; Markov semigroup; Hilbert space; L\'evy process; birth-and-death process}
\subjclass[2020]
{60J25, 60J57}

\begin{document}
\begin{abstract}
We provide a general framework for dual representations of Laplace transforms of Markov processes. Such representations state that the Laplace transform of a finite-dimensional distribution of a Markov process can be expressed in terms of a Laplace transform involving another Markov process, but with coefficients in the Laplace transform and time indices of the process interchanged. Dual representations of Laplace transforms have been used recently to study open ASEP \cite{Bryc_Wang_2018,bryc19limit} and to describe stationary measures of the open KPZ equation \cite{Bryc_et_al_2023}. Our framework covers both recently discovered examples in the literature and several new ones, involving general L\'evy processes and certain birth-and-death processes. 
\end{abstract}
\maketitle





\section{Introduction}\label{section_into}

Recently, several identities on Laplace transforms of Markov processes have been discovered in investigations of limit theorems. The first such example has appeared in \cite{Bryc_Wang_2018}.  Let $\mathbb B^{\textnormal {ex}}$ denote a normalized Brownian excursion \cite{revuz99continuous}, and let $Y$ denote the radial part of a three-dimensional Cauchy process (see \eqref{eq:Cauchy} below for more discussions). Then for $n\in\mathbb N, 0 =  \tilde s_0 <\tilde s_1<\cdots<\tilde s_n<\tilde s_{n+1}=1$ (we use $\tilde s_k$ to distinguish our convention of $s_k$ throughout the paper), $0=t_0<t_1<\cdots<t_n$, we have
\begin{equation}\label{eq:SPL}
\e\Big[ e^{- \sum\limits_{k=1}^{n} (t_k-t_{k-1}){\mathbb B}^{\textnormal{ex}}_{\tilde s_k}}  \Big]=
\sqrt{\frac2\pi}\int_0^{\infty}  \e\Big[ e^{-\frac{1}{2}\sum\limits_{k=0}^{n} (\tilde s_{k+1}-\tilde s_k)Y_{t_{k}}^2 
}| Y_{0}=y\Big] y^2  \d y,
\end{equation}
where the left-hand side is the joint Laplace transform of the Brownian excursion at times $\tilde s_1,\dots,\tilde s_n$, with coefficients $t_k-t_{k-1},k=1,\dots,n$, and the right-hand side is another joint Laplace transform (with infinite measure though) of the squared process $Y^2$ at times $t_1,\dots,t_n$, with coefficients $\tilde s_{k+1}-\tilde s_k, k=0,\dots,n$. In other words, the coefficients in the Laplace transform and time indices of the process  swap their roles. 

Since then, several identities of such a type as \eqref{eq:SPL} were discovered, with the processes $\mathbb B^{\textnormal {ex}}$ and $Y$ replaced by other processes, and also the measure on the right-hand side $\sqrt{2/\pi}y^2\d y$ modified accordingly. We refer to such identities as dual representations of Laplace transforms of the two Markov processes involved.
So far, all these dual representations have played important roles in developing scaling limit theorems for stochastic models with algebraic/combinatorial flavors, notably the asymmetric simple exclusion processes with open boundary (open ASEP) in steady state \cite{Bryc_et_al_2023,bryc19limit,bryc23asymmetric,corwin24stationary}. Proving limit theorems is not our concern here so we only briefly explain the approach behind it. For all these models, joint Laplace transforms of finite-dimensional distributions are explicit and their limit (after appropriate scaling) is computable. For open ASEP such a key feature is due to the so-called Derrida's matrix ansatz \cite{derrida06matrix,derrida93exact}, a well-known method in the mathematical physics literature, in combination with Markov representation in terms of Askey--Wilson processes \cite{bryc17asymmetric}. However, by a straightforward calculation, in the limit the coefficients of the Laplace transform and the time indices often swap their roles. In the case of Brownian excursion for example, the calculation leads to the right-hand side of \eqref{eq:SPL}. The translation from the right-hand side to the left-hand side of \eqref{eq:SPL} in some examples may not be straightforward, and sometimes can be quite challenging. For more details on how the dual representations play a role in limit theorems, we refer the readers to the aforementioned references.

The goal of this paper is to examine the dual representations of Laplace transforms from a different aspect. We are motivated by the following closely related questions:
\begin{itemize}
\item[(i)] What is the mathematical structure behind such representations? How can one conjecture and prove dual representations of Laplace transforms of Markov processes?
\item[(ii)] Are there other examples of dual representations, that is, ones involving other Markov processes?
\item[(iii)] Are there other applications of dual representations besides proving limit theorems?
\end{itemize}
For the first question, we present a relatively satisfactory answer. The main result of the paper is to provide a general framework explaining the origin of dual representations. We demonstrate that the crucial step in constructing dual representations of Laplace transforms is to show that the semigroups of two Markov processes are both diagonalized via certain unitary operators between suitable Hilbert spaces. The conditions are summarized in Assumption \ref{assump:1} below, and the main results are Theorem \ref{thm0} and Corollary \ref{corollary1}. This general framework then allows us to give an answer to question (ii):  in Section \ref{section_examples} we show how our general framework can explain all known instances of duality identities and discover new ones. In particular, 
the example of \eqref{eq:SPL} is explained in full detail in Section \ref{subsection_Brownian_excursion}. We also present two other examples, which are new in the literature: one concerns dual representations for pairs of L\'evy processes  in Section \ref{subsection_Levy_processes} and the other concerns dual representations between CIR diffusion and birth-and-death processes in Section \ref{subsection_CIR}.

For the third question, we are still at the stage of exploration. A promising direction is to consider, in place of the Laplace transform of finite-dimensional distributions of the Markov processes, the transform of integrals of the Markov processes, see Theorem \ref{thm1}. Some concrete examples involving L\'evy processes can be found in Section \ref{subsection_Levy_processes}. In particular, we  hope that these results might lead to new expressions for Laplace transforms of exponential functionals of L\'evy processes, see Corollary \ref{corollary3}.

Last but not least, we emphasize that the conditions presented in Assumption \ref{assump:1} may be somewhat relaxed. It is convenient to work with Hilbert spaces and a unitary operator between them. However,  the most consequential example so far does not fit into this framework. This is the example that appeared in the recent groundbreaking work on the stationary measures of so-called open KPZ equation \cite{corwin24stationary}, and the dual representation is established in \cite{Bryc_et_al_2023}, relating Yakubovich heat kernel and continuous dual Hahn process.  In Section \ref{section_kBM_and_dual_Hahn}, we provide a presentation of this example in a simplified version (see Remark \ref{rem:KPZ}), focusing on explaining how Assumption \ref{assump:1} might be relaxed. 


\section{Dual representation of Laplace transforms: a framework}\label{section_general_case}


Consider two Markov (or sub-Markov) processes $X=\{X_s\}_{s\in \bfS}$  and 
 $Y=\{Y_t\}_{t \in {\bfT}}$ with c\`adl\`ag paths, where $\bfS$ and $\bfT$ are two real intervals
 of the form $[0,b]$ ($b>0$) or $[0,\infty)$. We assume that the process $X$ (resp. $Y$) takes values in $\calX \subseteq \r$ (resp.~$\calY \subseteq \r$). Here each of the sets $\calX$ and $\calY$ is an interval (possibly infinite or semi-infinite) or a countable set.  Associated to processes $X$ and $Y$ are two Markov (or, possibly, sub-Markov) semigroups $\{{\mathcal P}_{s_1,s_2}\}_{s_1<s_2, s_i \in \bfS}$ and 
 $\{{\mathcal Q}_{t_1,t_2}\}_{t_1<t_2, t_i \in \bfT}$ defined via
 \begin{equation}\label{eqn:semigroup_X_Y}
 {\mathcal P}_{s_1,s_2} f(x)=\e[f(X_{s_2})| X_{s_1}=x], 
 \;\;\; 
 {\mathcal Q}_{t_1,t_2}g(y)=\e[g(Y_{t_2})| Y_{t_1}=y],
 \end{equation}
 where $f$ and $g$ are bounded Borel measurable functions on $\calX$ and $\calY$ respectively.

In the following assumption we list the key conditions that allow one to construct dual representations of Laplace transforms of Markov processes.
 
\begin{assumption}\label{assump:1}
There exist 
 
 \begin{itemize}
 \item[(i)]
 two Hilbert spaces $L^2(\calX,\mu)$ and $L^2(\calY,\nu)$,
and a unitary operator 
\[
\calF: L^2(\calX,\mu) \to L^2(\calY,\nu),
\]
where each $L^2$ space is the space of complex-valued functions, 
 \item[(ii)] two continuous functions $\phi: \calY \to \c^+$ and $\psi: \calX \to \c^+$, where $\c^+:=\{z\in \c \; : \; \re(z)\ge 0\}$ is the right complex half-plane,
 \item[(iii)] measurable functions $h_s(x): s \in \bfS, x\in \calX\to{(0,\infty)}$ and $j_t(y): t \in \bfT, y\in \calY\to{(0,\infty)}$, 
\end{itemize}
such that Markov semigroups of the processes $X$ and $Y$ satisfy
 \begin{align}
 \label{eqn:semigroup_X}
{\mathcal P}_{s_1,s_2} f(x)&=\left[\frac{1}{h_{s_1}} {\mathcal F}^{-1} e^{(s_1-s_2) \phi} {\mathcal F} (h_{s_2} f)\right](x), \;\;\; 
s_i \in {\bfS}, \; s_1<s_2, \\
\label{eqn:semigroup_Y}
{\mathcal Q}_{t_1,t_2}g(y)&=\left[\frac{1}{j_{t_1}} {\mathcal F} e^{(t_1-t_2) \psi} {\mathcal F}^{-1} (j_{t_2} g)\right](y), \;\;\; 
t_i \in {\bfT}, \; t_1<t_2, 
\end{align} 
 for all bounded Borel measurable functions $f\in L^2(\calX, (h_{s_2})^2\mu)$ and $g \in L^2(\calY, (j_{y_2})^2\nu)$. 
\end{assumption}

Formulas \eqref{eqn:semigroup_X} and 
\eqref{eqn:semigroup_Y} state that the Markov semigroup ${\mathcal Q}_{t_1,t_2}$ is a switch transform of the semigroup ${\mathcal P}_{s_1,s_2}$, see  \cite{Patie_2023}[Section 3.3].

One important implication of our Assumption \ref{assump:1} is that the Markov semigroups ${\mathcal P}_{s_1,s_2}$ and ${\mathcal Q}_{t_1,t_2}$ (which are originally defined via 
formulas \eqref{eqn:semigroup_X_Y} for bounded Borel measurable functions) can be uniquely extended to contraction operators acting on Hilbert spaces  \\
$L^2(\calX, (h_{s_2})^2\mu)$ and $L^2(\calY, (j_{y_2})^2\nu)$. Indeed, assuming that $f \in L^2(\calX, (h_{s_2})^2\mu)$ is bounded, we have 
\begin{equation}\label{eq:contraction2}
\left\|\calP_{s_1,s_2}f\right\|_{L^2(\calX,(h_{s_1})^2\mu)}\le \|f\|_{L^2(\calX,(h_{s_2})^2\mu)},
\end{equation}
which follows from \begin{equation}\label{eq:contraction}
\left\|\calF^{-1}e^{(s_1-s_2)\phi}\calF f\right\|_{L^2(\calX,\mu)} \le \|f\|_{L^2(\calX,\mu)},
\end{equation}  by our assumptions that $\calF$ is unitary and also that  $|e^{(s_{1}-s_{2}) \phi}|\le 1$. Then we use \eqref{eq:contraction2} to extend the domain of $\mathcal P_{s_1,s_2}$ to $L^2(\calX,(h_{s_2})^2\mu)$,
where the same inequality  
\eqref{eq:contraction2} 
holds. (For the extension, we set $\calP_{s_1,s_2}f$ as the $L^2$-limit of $\calP_{s_1,s_2}(f_n)$ as $n\to\infty$ with $f_n:=f{\bf 1}_{\{|f|\le n\}}, n\in\mathbb N$.)


\begin{remark}\label{rem:r}
The two formulas \eqref{eqn:semigroup_X} 
and \eqref{eqn:semigroup_Y} in the Assumption \ref{assump:1}(iii) should be viewed as spectral representations of the Markov semigroups $\calP_{s_1,s_2}$ and $\calQ_{t_1,t_2}$, extended to the corresponding Hilbert spaces. Sufficient conditions on when a homogeneous Markov semigroup can be extended from bounded Borel measurable functions to a certain Hilbert space are provided in Theorem 5.8 in \cite{Prato2006}: this result states that if the process $X$ is homogeneous with an invariant measure  $\mu$, then the Markov semigroup $\calP_{s_1,s_2}$ can be extended to a contraction semigroup in a Hilbert space. A similar argument can be used to extend this result to the case when the measure $\mu$ is excessive. Note that in the homogeneous case (when $h_s\equiv 1$), formula \eqref{eqn:semigroup_X} implies that $\calP_{s_1,s_2}$ is a normal operator. 
\end{remark}

\begin{remark}\label{rem:f}The reader should keep in mind that whenever they encounter expectations of the form 
\[
\varphi_f(x):=\mathbb E\big[ \xi \times f(X_{s_2}) | X_{s_1}=x\big],
\]
where $\xi$ is a bounded functional of the path $\{X_s\}_{s_1\le s \le s_2}$ and $f \in L^2(\calX,(h_{s_2})^2\mu)$, these should be understood as a result of a two-step procedure: first we extend the operator 
\[
f(\cdot) \mapsto {\mathbb E}\big[ \xi \times f(X_{s_2}) | X_{s_1}=\cdot\big]
\]
from bounded measurable functions $f$
 to the whole space $L^2(\calX,(h_{s_2})^2\mu)$ and then we apply this extension to a specific function $f$ in this Hilbert space. 
 For the extension, as before we take $f_n = f{\mathbf 1}_{\{|f|\le n\}}$, and with  $C$ satisfying $|\xi|\le C$ almost surely, we have
\[
\|\varphi_{f_n}-\varphi_{f_m}\|_{L^2(\calX,(h_{s_1})^2\mu)}  \le C \|f_n-f_m\|_{L^2(\calX,(h_{s_2})^2\mu)}.
\]
So $\{\varphi_{f_n}\}_{n\in\mathbb N}$ is Cauchy in $L^2(\calX,(h_{s_1})^2\mu)$ and we let $\varphi_f$ denote its limit. 
\end{remark}
We provide some comments on notations and conventions that will be used throughout the paper. In \eqref{eqn:semigroup_X}, $e^{(s_1-s_2)\phi}$ is understood as a multiplication operator mapping a measurable function $f(x)$ to $e^{(s_1-s_2)\phi(x)}f(x)$. The same is true for $e^{(t_1-t_2)\psi}$ in \eqref{eqn:semigroup_Y}. When $\calF$ is an operator and $f$ is in the domain of $\calF$, we write $\calF f(y) = [\calF f](y)$ where $y$ is in the domain of $\calF f$. 
For a sequence of numbers $\{u_k\}_{0\le k \le n}$ we denote the increments by 
$$
\Delta u_k:=u_{k+1}-u_k, \;\;\; 0\le k \le n-1.
$$
The following theorem is our first main result. 

\begin{theorem}\label{thm0}
Let conditions in Assumption \ref{assump:1} be true. Let $s_i \in \bfS$ and $t_i \in \bfT$
be such that 
$$
s_0<s_1<\dots<s_{n}, \;\;\; t_0<t_1< \dots < t_{n}. 
$$
Assume that $f : \calX \to \c$ and $g : \calY \to \c$ are functions such that $h_{s_n} f \in L^2(\calX,\mu)$ and 
\begin{equation}\label{condition_l_r}
j_{t_{n}} 
g={\mathcal F}(h_{s_{n}} f). 
\end{equation} 
Denote
\[
{\mathsf F}(x):=\e\Big[ e^{-\sum\limits_{k=0}^{n-1} \Delta t_k \psi(X_{s_k})
}
f(X_{s_{n}}) | X_{s_0}=x \Big], \;\;\; x \in \calX,
\]
and 
\[
{\mathsf G}(y):=\e\Big[ e^{-\sum\limits_{k=0}^{n-1}  \Delta s_k\phi(Y_{t_{k+1}})}g(Y_{t_{n}}) | Y_{t_0}=y \Big], 
\;\;\; y\in \calY. 
\]
Then $h_{s_0} {\mathsf F} \in L^2(\calX,\mu)$ and the functions ${\mathsf F}$ and ${\mathsf G}$ satisfy
\begin{equation}\label{eqn:thm0_main}
j_{t_0}{\mathsf G}={\mathcal F} (h_{s_0}{\mathsf F}) \in L^2(\calY,\nu). 
\end{equation}
\end{theorem}
\begin{proof}
We prove $h_{s_0}\mathsf F\in L^2(\calX,\mu)$. 
We compute 
\begin{align}
\nonumber
&h_{s_0}(x){\mathsf F}(x)=
h_{s_0}(x)\e\Big[ e^{\sum\limits_{k=0}^{n-1} (t_{k}-t_{k+1}) \psi(X_{s_k})} f(X_{s_{n}}) | X_{s_0}=x \Big]\\
\label{eq:thm0_proof1}
&=
\Big[h_{s_0} e^{(t_0-t_1) \psi} {\mathcal P}_{s_0,s_1} e^{(t_1-t_2) \psi} {\mathcal P}_{s_1,s_2} e^{(t_2-t_3) \psi} \cdots
\\
\nonumber
& \qquad \qquad\cdots   {\mathcal P}_{s_{n-2},s_{n-1}} e^{(t_{n-1}-t_{n}) \psi} 
{\mathcal P}_{s_{n-1},s_{n}} f\Big](x) \\
\nonumber
&=\Big[e^{(t_0-t_1) \psi}{\mathcal F}^{-1} e^{(s_0-s_1) \phi} {\mathcal F} 
e^{(t_1-t_2) \psi}  {\mathcal F}^{-1} e^{(s_1-s_2) \phi}  \cdots  \\
\nonumber
& \qquad \qquad\cdots  {\mathcal F}^{-1} e^{(s_{n-2}-s_{n-1}) \phi} {\mathcal F} 
e^{(t_{n-1}-t_{n}) \psi} 
{\mathcal F}^{-1} e^{(s_{n-1}-s_{n}) \phi} {\mathcal F} 
(h_{s_{n}} f)\Big](x).   
\end{align}
We have seen that under our assumption,  $h_{s_n}f\in L^2(\calX,\mu)$ implies 
$$
{\mathcal F}^{-1} e^{(s_{n-1}-s_{n}) \phi} {\mathcal F} (h_{s_{n}} f)\in L^2(\calX,\mu)$$
 (by contraction \eqref{eq:contraction}). The multiplication by $e^{(t_1-t_2)\psi}$ is also a contraction by our assumption. 
By repeating the argument $n-1$ times, we see that 
$h_{s_0}{\mathsf F} \in L^2(\calX,\mu)$.  Using the assumption \eqref{condition_l_r} and the above formula \eqref{eq:thm0_proof1}, we compute 
\begin{align}
\nonumber
\frac{1}{j_{t_0}} {\mathcal F} (h_{s_0} {\mathsf F})
&=
 \frac{1}{j_{t_0}} {\mathcal F} 
e^{(t_0-t_1) \psi}  {\mathcal F}^{-1} e^{(s_0-s_1) \phi} 
{\mathcal F} 
e^{(t_1-t_2) \psi}  {\mathcal F}^{-1} e^{(s_1-s_2) \phi} 
  \\
\label{eq:thm0_proof2}
& \qquad  \cdots  {\mathcal F}^{-1} e^{(s_{n-2}-s_{n-1}) \phi} {\mathcal F} 
e^{(t_{n-1}-t_{n}) \psi} 
{\mathcal F}^{-1} e^{(s_{n-1}-s_{n}) \phi} 
(j_{t_{n}} g)   \\
\nonumber
&=
{\mathcal Q}_{t_0,t_1} e^{(s_0-s_1) \phi}
 {\mathcal Q}_{t_1,t_2} e^{(s_1-s_2) \phi} \cdots 
  {\mathcal Q}_{t_{n-1},t_{n}}  e^{(s_{n-1}-s_{n}) \phi} g.
  \end{align}
  In the last step, we used $e^{(s_{n-1}-s_n)\phi}(j_{t_n}g) = j_{t_n}(e^{(s_{n-1}-s_n)\phi}g)$, exchanging the order of two multiplicative operators. 
We recognize the last expression above as a function of $y$ as $\e\Big[ e^{\sum\limits_{k=0}^{n-1} (s_{k}-s_{k+1}) \phi(Y_{t_{k+1}})} g(Y_{t_{n}}) | Y_{t_0}=y \Big]= {\mathsf G}(y)$, as desired. 
\end{proof}

The next corollary presents a useful mechanism for obtaining dual representations of Laplace transforms when the processes $X$ and $Y$ are started at certain initial distributions. 

\begin{corollary}\label{corollary1}
Under the assumptions of Theorem \ref{thm0}, assume in addition  that ${\ell}^X_{\rm init}: \calX \to \c$ 
and ${\ell}^Y_{\rm init} : \calY \to \c$ are functions such that 
${\ell}^X_{\rm init}/h_{s_0} \in L^2(\calX, \mu)$ and
\begin{equation}\label{eqn:l_r_condition}
    \frac{{\ell}^Y_{\rm init}}{j_{t_0}}=
{\mathcal F} \Big(\frac{{\ell}^X_{\rm init}}{h_{s_0}}\Big).
\end{equation}
Then 
\begin{multline}\label{eq:duality_with_entrance_law}
\int_{\calX} 
\e\Big[ e^{-\sum\limits_{k=0}^{n-1}  \Delta t_k\psi(X_{s_k})
}
f(X_{s_{n}}) | X_{s_0}=x \Big] {\ell}^X_{\rm init}(x) \mu(\d x)\\=
\int_{\calY} \e\Big[ e^{-\sum\limits_{k=0}^{n-1}  \Delta s_k\phi(Y_{t_{k+1}})}g(Y_{t_{n}}) | Y_{t_0}=y \Big]
{\ell}^Y_{\rm init}(y) \nu(\d y).
\end{multline}
\end{corollary}
\begin{proof}
Formula \eqref{eq:duality_with_entrance_law} can be obtained from \eqref{eqn:thm0_main} by Plancherel's theorem (the fact that ${\mathcal F}$ is an isometry): 
\begin{align*}
   \int_{\calX} {\mathsf F}(x) {\ell}^X_{\rm init}(x) \mu(\d x)&=
   \int_{\calX} h_{s_0}(x) {\mathsf F}(x) \frac{{\ell}^X_{\rm init}(x)}{h_{s_0}(x)} \mu(\d x)=
   \int_{\calY} [{\mathcal F}(h_{s_0} {\mathsf F})](y)
  \Big[  {\mathcal F} \Big(\frac{{\ell}^X_{\rm init}}{h_{s_0}}\Big) \Big](y) \nu(\d y)\\
    &=
    \int_{\calY} j_{t_0}(y) {\mathsf G}(y) \frac{{\ell}^Y_{\rm init}(y)}{j_{t_0}(y)} \nu(\d y)=
     \int_{\calY} {\mathsf G}(y) {\ell}^Y_{\rm init}(y) \nu(\d y).
\end{align*}
\end{proof}

We would like to point out that the functions $f$, $g$, ${\ell}^X_{\rm init}$ and ${\ell}^Y_{\rm init}$
that appear in Theorem \ref{thm0} and 
Corollary \ref{corollary1} may depend on parameters $\{s_k\}_{0\le k \le n}$ and $\{t_k\}_{0\le k \le n}$. 
 When ${\ell}^X_{\rm init}$ and ${\ell}^Y_{\rm init}$ are positive functions, formula \eqref{eq:duality_with_entrance_law} can be interpreted as a dual representation of Laplace transforms as follows: in the left-hand side we have an expected value of a functional of $X$ process
\[
e^{-\sum\limits_{k=0}^{n-1} \Delta t_k \psi(X_{s_k})
}
f(X_{s_{n}}), 
\]
where $X$ is started at time $s_0$ at a (possibly infinite) measure ${\ell}^X_{\rm init}(x) \mu(\d x)$, and in the right-hand side we have an expected value of a functional of $Y$ process
\[
e^{-\sum\limits_{k=0}^{n-1}  \Delta s_k\phi(Y_{t_{k+1}})}g(Y_{t_{n}}),
\]
where $Y$ is started at time $t_0$ at a (possibly infinite) measure ${\ell}^Y_{\rm init}(y)\nu(\d y)$. If one takes $f=1$, then the left-hand side of \eqref{eq:duality_with_entrance_law} becomes exactly the Laplace transform of $(X_{s_1},\dots,X_{s_n})$ with kernel $\calP_s$ and initial law $\ell_{\textnormal{init}}^X$, and to obtain the identity one computes on the right-hand side $\ell_{\textnormal{end}}^Y$ and $g$ explicitly via conditions \eqref{condition_l_r} and \eqref{eqn:l_r_condition}.  The example \eqref{eq:SPL} mentioned in Section \ref{section_into} is obtained in such a way, as explained in Section \ref{subsection_Brownian_excursion}.

Next, we express the result in Theorem \ref{thm0} in an integral form, which would be more useful for certain applications. We assume $h_s = j_t = 1$, so that the Markov processes are homogeneous (and we then let $\mathcal P_s$ and $\mathcal Q_t$  denote the corresponding semigroups). Recall that a $\sigma$-finite measure $\mu$ is excessive for the Markov process $X$, if $\mu\mathcal P_s\le \mu$ for all $s>0$. 

\begin{theorem}\label{thm1}
Let conditions in Assumption \ref{assump:1} be true and assume in addition that $h_{s}=j_t=1$ and the measure $\mu$ (respectively, $\nu$) is excessive for the process $X$ (respectively, $Y$). Let $s=s(w)  : [0,1] \to [0,\mathfrak s]$ and $t=t(w)  : [0,1] \to [0,\mathfrak t]$ be two increasing  
and right-continuous functions with $s(0) = t(0) = 0, s(1) = \mathfrak s$ and $t(1) = \mathfrak t$. 
Assume that $f : \calX \to \c$ and $g : \calY \to \c$ are functions such that $f \in L^2(\calX,\mu)$ and $g={\mathcal F}f$. 
Denote 
\begin{align*}
{\mathsf F}(x)& :=\e\Big[ e^{- \int_{(0,1]}  \psi(X_{s(w-)}) \d t(w)}
f(X_{\mathfrak s}) | X_0=x \Big], \;\;\; x \in \calX,\\
{\mathsf G}(y)&:=\e\Big[ e^{- \int_{(0,1]}   \phi(Y_{t(w)}) \d s(w)} g(Y_{\mathfrak t}) | Y_0=y \Big], 
\;\;\; y\in \calY. 
\end{align*}
Then ${\mathsf F} \in L^2(\calX,\mu)$ and the functions ${\mathsf F}$ and ${\mathsf G}$ satisfy $\mathsf G = \calF \mathsf F$.
\end{theorem}

\begin{proof}
We take an integer $n\ge 2$ and define $s_{n,k}=s(k/n)$ and $t_{n,k}=t(k/n)$ for $0\le k \le n$. We denote 
\begin{align*}
{\mathcal I}_n&:=\sum\limits_{k=0}^{n-1}  \psi(X_{s_{n,k}})
\Delta t_{n,k} , \;\; {\mathcal I}:=\int_{(0,1]}  \psi(X_{s(w-)}) \d t(w),\\ 
 {\mathcal J}_n&:=\sum\limits_{k=0}^{n-1}  \phi(Y_{t_{n,k+1}})\Delta s_{n,k}, \;\;
{\mathcal J}:=\int_{(0,1]} \phi(Y_{t(w)}) \d s(w). 
\end{align*}

Due to our assumptions, the functions 
$\phi(Y_{t(w)})$ and $\psi(X_{s(w)})$ are c\`adl\`ag functions, thus we have ${\mathcal I}_n \to {\mathcal I}$ and ${\mathcal J}_n \to {\mathcal J}$ as $n\to +\infty$ almost surely. 
The rest of the proof is devoted to extending certain operators to Hilbert spaces. 
We denote 
\begin{align*}
{\mathcal P}^{(n)} f(x)&=\e\Big[ e^{-{\mathcal I}_n}
f(X_{\mathfrak s}) | X_0=x \Big],\\
{\mathcal Q}^{(n)} g(y)&=\e\Big[ e^{-{\mathcal J}_n}g(Y_{\mathfrak t}) | Y_0=y \Big],
\end{align*}
and 
\begin{align*}
{\mathcal P} f(x)&=\e\Big[ e^{-{\mathcal I}}
f(X_{\mathfrak s}) | X_0=x \Big] = \mathsf F(x),\\
{\mathcal Q} g(y)&=\e\Big[ e^{-{\mathcal J}}g(Y_{\mathfrak t}) | Y_0=y \Big] = \mathcal Q\calF f(y). 
\end{align*}
The above expectations are well-defined for bounded measurable functions $f$ and $g$ (the random variables ${\mathcal I}_n$, ${\mathcal J}_n$,
${\mathcal I}$ and ${\mathcal J}$
 have positive real parts).  
 
Our goal is to show first that $\mathcal P$ can be extended to an operator on $L^2(\calX,\mu)$ and then to show $\mathsf G = \calF \mathsf F$. For the second step, it suffices to show the intertwining relation (see \cite{Patie_et_al_2019})
 \[
 \calF \mathcal P = \mathcal Q \calF
 \]
 on $L^2(\calY,\nu)$. 

We first recall the following facts: 
\begin{itemize}
 \item[(i)] ${\mathcal P}^{(n)}$ can be extended as an operator 
on $L^2(\calX, \mu)$ which satisfies $\|{\mathcal P}_n f\|_2\le \|f\|_2$ for all $f\in L^2(\calX, \mu)$ (here the norm comes from $L^2(\calX, \mu)$). This can be seen by the argument presented in Remark \ref{rem:f} with 
$\xi=\exp(-{\mathcal I}_n)$ and $C=1$. 
Similar property holds for ${\mathcal Q}^{(n)}$ and $L^2(\calY, \nu)$. \item[(ii)] For any bounded $f$ and every $x\in \calX$ we have 
\[
{\mathcal P}^{(n)} f(x) \to {\mathcal P} f(x), \;\;\; n\to +\infty,
\]
and similarly for every bounded $g$ and every $y\in \calY$ 
$$
{\mathcal Q}^{(n)} g(y) \to {\mathcal Q} g(y), \;\;\; n\to +\infty. 
$$
This follows from the Dominated Convergence Theorem. 
\end{itemize}
Now take $f\in L^2(\calX, \mu)$ and for $l>0$ define $f_l(x)=f(x){\mathbf 1}_{\{|f(x)|\le l\}}$. In the rest of the proof we write $\|\cdot\|_2 = \|\cdot\|_{L^2(\calX,\mu)}$. Applying triangle inequality and property (i) above, we get
$$
\| {\mathcal P}^{(n)} f-{\mathcal P}^{(m)} f\|_2 \le \| {\mathcal P}^{(n)} f_l-{\mathcal P}^{(m)} f_l\|_2+2 \|f-f_l\|_2.
$$
The second term $\|f-f_l\|_2 $ converges to zero as $l\to +\infty$. 
Next, we estimate
\begin{align*}
\| {\mathcal P}^{(n)} f_l-{\mathcal P}^{(m)} f_l\|^2_2 &=
\int_{\calX}  \e\Big[ (e^{-{\mathcal I}_n}-e^{-{\mathcal I}_m})
f_l(X_{\mathfrak s}) | X_0=x \Big]^2 \mu(\d x)\\
&\le 
\int_{\calX}  \e\Big[ (e^{-{\mathcal I}_n}-e^{-{\mathcal I}_m})^2 |X_0=x\Big]
{\mathbb E}\Big[f_l(X_{\mathfrak s})^2 | X_0=x \Big] \mu(\d x).
\end{align*}
The function $x\mapsto {\mathbb E}[f_l(X_{\mathfrak s})^2 | X_0=x \Big]$ is in $L^1(\calX, \mu)$, due to our assumption that the measure $\mu$ is excessive: 
$$
\int_{\calX} {\mathbb E}[f_l(X_{\mathfrak s})^2 | X_0=x \Big] \mu(\d x)
\le \int_{\calX} f_l(x)^2 \mu(\d x) \le \|f\|_2^2. 
$$
Since $e^{-{\mathcal I}_n} \to e^{-{\mathcal I}}$ as $n\to \infty$ almost surely and all these random variables are bounded by 1, the function $x\to  \e[ (e^{-I_n}-e^{-I_m})^2 |X_0=x]$ is bounded by one and converges almost everywhere to zero as both $m,n\to +\infty$. 
By the Dominated Convergence Theorem, we conclude that 
$$
\int_{\calX}  \e\Big[ (e^{-{\mathcal I}_n}-e^{-{\mathcal I}_m})^2 |X_0=x\Big]
{\mathbb E}\Big[f_l(X_{\mathfrak s})^2 | X_0=x \Big] \mu(\d x)
$$
converges to zero as $m,n\to +\infty$. This shows that the sequence of $L^2(\calX,\mu)$ functions ${\mathcal P}^{(n)} f$ is Cauchy and thus converges in $L^2(\calX,\mu)$.  Thus we have shown that ${\mathcal P}^{(n)}$ converges to some operator $\tilde {\mathcal P}$ 
 in strong operator topology. The statement in item (ii) above implies that 
for every bounded $f \in L^2(\calX,\mu)$ we have 
${\mathcal P}^{(n)} f(x) \to {\mathcal P} f(x)$. Since such bounded functions are dense in $L^2(\calX, \mu)$, it means that 
$\tilde {\mathcal P}$ is the unique extension of ${\mathcal P}$ to $L^2(\calX, \mu)$, and we will denote this extension by the same symbol ${\mathcal P}$. 
The same argument shows that ${\mathcal Q}^{(n)}$ converges (in strong operator topology) to  the unique extension of the operator ${\mathcal Q}$ on $L^2(\calY, \nu)$. Thus we have obtained  extensions of the operators ${\mathcal P}$ and ${\mathcal Q}$ to the corresponding Hilbert spaces. 
Then,  for every $f \in L^2(\calX, \mu)$ by taking the limit of both sides of ${\mathcal F} {\mathcal P}^{(n)} f={\mathcal Q}^{(n)} {\mathcal F} f$ (the latter statement follows from Theorem 
\ref{thm0}),  we have the intertwining relationship ${\mathcal F} {\mathcal P} f={\mathcal Q} {\mathcal F} f$.
\end{proof}

\section{Examples}\label{section_examples}


In this section we present examples to illustrate the general results presented in 
Theorems \ref{thm0} and \ref{thm1}. We discuss L\'evy processes, Brownian excursions and meanders, and an example connecting the CIR diffusion process with a birth and death process. 


\subsection{L\'evy processes}\label{subsection_Levy_processes}


We consider a L\'evy process $X=\{X_s\}_{s\ge 0}$ with characteristic exponent $\Phi$, that is
$$
\e[e^{\i y X_s}\vert X_0 = 0]=e^{-s \Phi(y)}, \quad s\ge 0, y\in\mathbb R,
$$
and another L\'evy process $Y=\{Y_t\}_{t\ge 0}$ with characteristic exponent $\Psi$. Note that in our discussions a L\'evy process may start from a non-zero fixed value.
We take $L^2(\calX,\mu) = L^2(\calY,\nu) = L^2(\mathbb R,\d x)$, $\bfS=\bfT=[0,\infty)$ and consider a unitary operator ${\mathcal F}$ on the Hilbert space $L^2({\mathbb R}, \d x)$ given by 
the Fourier transform 
$$
{\mathcal F}f(z)=\frac{1}{\sqrt{2\pi}} \int_{\r} f(x) e^{-\i z x} \d x.
$$
 More precisely, the Fourier transform operator ${\mathcal F}$ is initially defined for Schwartz functions $f$ via the above integral formula
and then it is extended as a unitary operator on $L^2(\r, \d x)$. 
We set $h_s(x)=j_t(y)=1$ for all $s,t\ge 0$ and $x,y\in {\mathbb R}$. 
It is well-known that the transition operators of a L\'evy process can be extended to a contraction semigroup on $L^2({\mathbb R}, \d x)$ and that the Fourier transform diagonalizes these transition operators: 
$$
\calP_sf(x) = \e[f(X_s)|X_0=x]=\e[f(x+X_s)|X_0=0]=\big[{\mathcal F}^{-1} e^{- s \Phi} {\mathcal F}  f\big] (x), \;\;\; f\in L^2({\mathbb R}, \d x).
$$
For the derivation of the above equality, see  \cite{Jacob2001} or Exercise 
3.4.3 in \cite{Applebaum_2009}. We emphasize that the expectations such as  $\e[f(X_s)|X_0=x]$ are understood in the sense described in Remark \ref{rem:f}.
Let us denote by $\hat Y$ the dual process of $Y$, that is $
\hat Y_t=-Y_t$. We have the following identity for the dual process 
$$
\calQ_tg(y) = \e[g(\hat Y_t)| \hat Y_0=y]=\e[g(-y-Y_t)|Y_0=0]=\big[{\mathcal {F}} e^{- t \Psi} {\mathcal F}^{-1} g\big](y), \;\;\; g\in L^2({\mathbb R}, \d x).
$$
This can be checked by noting that ${\mathcal F}^{-1}={\mathcal {S F}}={\mathcal {F S}}$, where ${\mathcal S}$ the operator ${\mathcal S} f(x)=f(-x)$. Now we can state and prove the following result.

\begin{proposition}\label{prop_Levy1}
 Assume that $s=s(w)  : [0,1] \to  [0,\infty)$ and $t=t(w)  : [0,1] \to [0,\infty)$ are two increasing and continuous from the right differentiable functions, such that $s(0)=t(0)=0$. For $f \in L^2({\mathbb R}, \d x)$ we define
 \begin{align}
     \label{def_R_Levy}
     {\mathsf F}(x) & =\e\Big[ e^{- \int_{(0,1]} \Psi(X_{s(w-)})\d t(w)} f( X_{s(1)})| X_0=x\Big],\\
     \nonumber
    {\mathsf G}(y) & =\e \Big[ e^{- \int_{(0,1]} \Phi(\hat Y_{t(w)})\d s(w)}
{\mathcal F} f(\hat Y_{t(1)}) | \hat Y_0=y\Big].
\end{align}     
Then functions ${\mathsf F}$ and ${\mathsf G}$ are in  $L^2({\mathbb R}, \d x)$ and $
{\mathsf G}={\mathcal F} {\mathsf F}$.
\end{proposition}
\begin{proof}
We apply Theorem \ref{thm1} and use the fact that the Lebesgue measure is invariant for any L\'evy process. 
\end{proof}

\begin{remark}
We would like to present an alternative method for deriving the result of Proposition \ref{prop_Levy1} in the case when both $s(w)$ and $t(w)$ are differentiable functions and $f$ is a Schwartz class function.  We will only sketch the main ideas in the argument, without trying to justify all the steps. For $v\in [0,1]$ and $x \in \r$ we define 
\begin{equation*}
     {\mathsf F}(v,x)=\e\Big[ e^{- \int_0^v \Psi(X_{s(w)})\d t(w)} f( X_{s(v)})| X_0=x\Big].
 \end{equation*}
The function ${\mathsf F}(v,y)$ solves the backward Kolmogorov equation 
\begin{equation}\label{eq:PIDE_for_R}
\partial_v {\mathsf F}=-(t'(v) \Psi + s'(v) {\mathcal L}_X ) {\mathsf F},
\end{equation}
with the initial condition ${\mathsf F}(0,x)=f(x)$. Here ${\mathcal L}_X$ is the Markov generator of the L\'evy process $X$.  It is known that the operator ${\mathcal L}_X$ can be diagonalized via Fourier transform: 
$$
{\mathcal L}_X={\mathcal F}^{-1} \Phi {\mathcal F},
$$
thus the partial integro-differential equation (PIDE) \eqref{eq:PIDE_for_R} can be written in the form 
$$
\partial_v \mathsf F=-(t'(v) \Psi + s'(v) {\mathcal F}^{-1} \Phi {\mathcal F}) \mathsf F. 
$$
Defining ${\mathsf G}(v,y)={\mathcal F} [{\mathsf F}(v,\cdot)](y)$ (here the Fourier transform acts on $x$-variable) and recalling that ${\mathcal L}_{\hat Y}={\mathcal F} \Psi {\mathcal F}^{-1}$  we obtain from the above equation
\begin{equation}
\label{eq:PIDE_for_L}
\partial_v {\mathsf G}=-(t'(v) {\mathcal F} \Psi {\mathcal F}^{-1} + s'(v)  \Phi ) {\mathsf G}=
-( t'(v) \mathcal L_{\hat Y}+s'(v) \Phi ){\mathsf G}. 
\end{equation}
We recognize \eqref{eq:PIDE_for_L} as a backward Kolmogorov equation 
whose solution, with the initial condition ${\mathsf G}(0,y)={\mathcal F} f(y)$, is given by 
 \begin{equation*}
    {\mathsf G}(v,y)=\e \Big[ e^{- \int_0^v \Phi(\hat Y_{t(w)})\d s(w)}
{\mathcal F} f(\hat Y_{t(v)}) | \hat Y_0=y\Big],
\end{equation*}   
and we recover the statement of Proposition 
\ref{prop_Levy1}. 

We would like to emphasize that the above is not a rigorous proof (at the very least, one would need to ensure the uniqueness of solutions of the PIDEs and justify the interchange of partial derivative $\partial_v$ and Fourier transform), however, we feel that this discussion merits inclusion because it highlights the ideas involved in duality identities. 
\end{remark}

As an application of Proposition \ref{prop_Levy1}, we derive an expression for the joint Laplace--Fourier transform 
$$
{\mathbb E}\big[  e^{\i \alpha X_s+\i \beta L_1(s)-\lambda L_2(s)} \vert X_0=0\big], 
$$
of the triple $(X_s,L_1(s),L_2(s))$, where $X$ is a L\'evy process with characteristic exponent $\Phi$ and 
$$
L_k(s):=\int_0^s X_u^k \; \d u, \quad k=1,2 \;\; {\textnormal{ and }} \;s>0.
$$
We define the kernel $p_t(y_1,y_2)$  via 
$$
\e\Big[  e^{-\int_0^t \Phi(\sqrt{2\lambda} W_u - \beta u) \d u} 
{\mathbf 1}_{\{ \sqrt{2\lambda} W_t - \beta t \in A\}} |W_0=y_1 \Big]
= \int_A p_t(y_1,y_2) \d y_2,
$$
where $W$ is the standard Brownian motion and $A$ a Borel set of $\mathbb R$. In the case when the L\'evy process $X$ is symmetric so that $\Phi$ is a positive function on $\r$, we can identify $p_t(y_1,y_2)$ as the transition probability density of the Brownian motion with drift 
$\{\sqrt{2\lambda}W_t - \beta t\}_{t\ge 0}$ killed at rate $\Phi(y)$. Note that the kernel $p_t$ depends on both $\beta$ and $\Phi$.

In the following two corollaries we need to use identity ${\mathsf F}={\mathcal F}^{-1} {\mathsf G}$ in pointwise sense. We justify it as follows: we show that ${\mathsf G} \in L_1(\r, \d y)$ and ${\mathsf F}$ is continuous. Therefore ${\mathcal F}^{-1} {\mathsf G}$ is also continuous and we have two continuous functions that are equal almost everywhere -- thus they are equal for every $x$. 

\begin{corollary}\label{corollary2} 
For $\alpha, \beta \in \r$ and $s, \lambda >0$ we have
\begin{equation}\label{corollary2:eq1}
{\mathbb E}\big[  e^{\i \alpha X_s+\i \beta L_1(s)-\lambda L_2(s)} \vert X_0=0\big]=\int_{\r} p_s(y,\alpha) \d y. 
\end{equation}
\end{corollary}
\begin{proof}
We take $Y_t=-\sqrt{2\lambda}W_t + \beta t$, so that 
$\hat Y_t = \sqrt{2\lambda}W_t - \beta t$ and 
$
\Psi(x)=\lambda x^2 - \i \beta x.
$
We fix an arbitrary Schwartz class function $g(y)$ and set $f={\mathcal F}^{-1} g$  and we also take $s(w)=t(w)=vw$. 
Denote
$$
{\mathsf F}(x):=\e\Big[e^{-\lambda L_2(v)+\i\beta L_1(v)}f(X_v)|X_0=x\Big]
$$
and 
$$
{\mathsf G}(y):=\e\Big[  e^{-\int_0^v \Phi(\sqrt{2\lambda} W_u - \beta u) \d u} 
g(\sqrt{2\lambda} W_v - \beta v) |W_0=y \Big]. 
$$
Proposition \ref{prop_Levy1} implies ${\mathsf F}= {\mathcal F}^{-1} {\mathsf G}$. Next we show that ${\mathsf F}$ is a contiuous function 
and ${\mathsf G} \in L^1 (\r ,\d y)$, so that the identity ${\mathsf F}= {\mathcal F}^{-1} {\mathsf G}$ holds pointwise, in particular it gives us
\begin{equation}\label{eq115}
{\mathsf F}(0)=\frac{1}{\sqrt{2\pi}} \int_{\r} {\mathsf G}(y) \d y. 
\end{equation}
To show that ${\mathsf F}$ is continuous, we use translation invariance of the L\'evy process $X$ and write 
$$
{\mathsf F}(x)=\e\Big[e^{-\lambda \int_0^v (x+X_u)^2 \d u +\i\beta (x+\int_0^v X_u \d u)}f(x+X_v)|X_0=0\Big]
$$
The term 
$$
e^{-\lambda \int_0^v (x+X_u)^2 \d u +\i\beta (x+\int_0^v X_u \d u)}
$$
is continuous in $x$ and bounded by one in absolute value. Similarly $f(x+X_v)$ is continuous in $x$ and bounded by some constant (since $f$ is Schwartz function -- thus bounded and continuous). By the dominated convergence theorem we get (for any $x_0 \in \r$): 
${\mathsf F}(x) \to {\mathsf F}(x_0)$ as $x\to x_0$, so ${\mathsf F}$ is  continuous. 

We deal with the function ${\mathsf G}$ in a similar way: we write
$$
{\mathsf G}(y)=\e\Big[  e^{-\int_0^v \Phi(\sqrt{2\lambda} (y+W_u) - \beta u) \d u} 
g(\sqrt{2\lambda} (y+W_v) - \beta v) |W_0=0 \Big],
$$
use the fact that $\Phi$ has a positive real part, so that the term
$$
e^{-\int_0^v \Phi(\sqrt{2\lambda} (y+W_u) - \beta u) \d u}
$$
is bounded by one in absolute value, and then compute (applying Fubini Theorem)
\begin{align*}
\int_{\r} |{\mathsf G}(y)| \d y 
&\le \int_{\r} \e\Big[  
|g(\sqrt{2\lambda} (y+Z) - \beta v)| \Big] \d y\\
&=\e\Big[  \int_{\r} |g(\sqrt{2\lambda} (y+Z) - \beta v) \d y \Big]= \frac{1}{\sqrt{2\lambda}} \int_{\r} g(y) \d y.
\end{align*}
Here we denoted by $Z$ a $N(0,v)$ random variable. Since $g$ is in Schwartz class, it is in $L^1(\r, \d x)$ and we see that ${\mathsf G}$ is also in $L^1(\r, \d x)$. 

Next, we introduce 
$$
F_v(\alpha,\beta,\lambda):={\mathbb E}\big[  e^{\i \alpha X_v+\i \beta L_1(v)-\lambda L_2(v)} \vert X_0=0\big].
$$
 From \eqref{def_R_Levy} we find 
\begin{align}\label{corollary2_eqn2}
{\mathsf F}(0) & = \e\Big[e^{-\lambda L_2(v)+\i\beta L_1(v)}f(X_v)|X_0=0\Big]\nonumber\\
& = \e\Big[ e^{-\lambda L_2(v)+\i\beta L_1(v)}\frac1{\sqrt{2\pi}}\int_\r e^{\i yX_v}g(y)\d y|X_0 = 0\Big] = \frac{1}{\sqrt{2\pi}} \int_{\r} F_v(y,\beta,\lambda)  g(y) \d y,
\end{align}
where in the last step we applied Fubini's theorem.
On the other hand, from \eqref{eq115} we have 
\begin{align}
\nonumber
{\mathsf F}(0)&=\frac{1}{\sqrt{2\pi}} \int_{\r} 
\e\Big[  e^{-\int_0^v \Phi(\sqrt{2\lambda} W_u - \beta u) \d u} 
g(\sqrt{2\lambda} W_v - \beta v) |W_0=y_1 \Big] \d y_1\\
\label{corollary2_eqn3}
&=\frac{1}{\sqrt{2\pi}} \int_{\r} \int_{\r} p_v(y_1,y_2) g(y_2) \d y_1 \d y_2.
\end{align}
Since the right-hand sides in both \eqref{corollary2_eqn2} and 
\eqref{corollary2_eqn3} are equal for all Schwartz functions $g$, we conclude that \eqref{corollary2:eq1} holds. 
\end{proof}

We also use Proposition \ref{prop_Levy1} to obtain the following result about exponential functionals of L\'evy processes.

\begin{corollary}\label{corollary3} Let $N=\{N_t\}_{t\ge 0}$ be a standard Poisson process. For any L\'evy process $X$ with characteristic exponent $\Phi$ and any $\alpha \in \r$  and $\lambda,v>0$ we have
 \begin{equation}\label{eq:Laplace_transform_exp_functional}
\e\Big[  e^{\lambda\int_0^v e^{\i \alpha  X_u
 }\d u} |X_0=0 \Big]
=e^{\lambda v}
\sum\limits_{n\ge 0}
\e\Big[  e^{- \frac{1}{\lambda}\int_0^{\lambda v} \Phi(\alpha (n- N_u)) \d u} 
{\mathbf 1}_{\{N_{\lambda v}=n\}}|N_0=0 \Big].
\end{equation}
\end{corollary}
\begin{proof}
We will prove \eqref{eq:Laplace_transform_exp_functional} in the special case $\alpha=1$. The general case can be obtained by scaling $X \mapsto \alpha X$.
 We take $Y_t=N_t$, $s(w)=vw$ and $t(w)=\lambda v w$ where $\lambda$ and $v$ are positive numbers. Let
    $$
    f(x)=f_{\epsilon}(x):=e^{- \frac{\epsilon}{2} x^2},
    $$
    for $\epsilon>0$. 
 We have $\Psi(x)=1-\exp(\i x)$ as the characteristic exponent of $Y$, and 
    $$
    g(y)=g_{\epsilon}(y):={\mathcal F} f_{\epsilon}(y)=\frac{1}{\sqrt{\epsilon}}e^{-\frac{y^2}{2\epsilon}}.
    $$
    From Proposition \ref{prop_Levy1} we find 
    $$
    {\mathsf F}(x)=e^{-\lambda v} 
    \e\Big[ e^{\lambda \int_0^v e^{\i X_u} \d u} f_{\epsilon}( X_v) | X_0=x\Big]
    $$
    and 
    \begin{align*}
        {\mathsf G}(y)&=\e \Big[ e^{-  \frac{1}{\lambda} \int_0^{\lambda v} \Phi(- N_u)\d u}
g_{\epsilon}(-N_{\lambda v}) | N_0=-y\Big]\\
&= \sum\limits_{n\ge 0} g_{\epsilon}(y-n)
\e \Big[ e^{-  \frac{1}{\lambda} \int_0^{\lambda v} \Phi(-N_u)\d u}
{\mathbf 1}_{\{N_{\lambda v}=n-y\} }| N_0=-y\Big]
= \sum\limits_{n\ge 0} g_{\epsilon}(y-n) H_n(y),
    \end{align*}
    where we denoted 
    $$
    H_n(y):=\e \Big[ e^{-  \frac{1}{\lambda} \int_0^{\lambda v} \Phi(y-N_u)\d u}
{\mathbf 1}_{\{N_{\lambda v}=n\} }| N_0=0\Big].
    $$ 
    As in the proof of Corollary \ref{corollary2}, one can show that ${\mathsf F}$ is continuous and ${\mathsf G} \in L^1(\r,\d y)$, so that the statement ${\mathsf F}= {\mathcal F}^{-1} {\mathsf G}$ (which follows from Proposition \ref{prop_Levy1}) holds pointwise. Thus we obtain 
\[
    {\mathsf F}(0)=\frac{1}{\sqrt{2\pi}} \int_{\mathbb R} {\mathsf G}(y) \d y=
    \sum_{n\ge 0} \frac{1}{\sqrt{2\pi}}
    \int_{\mathbb R}g_{\epsilon}(y-n) H_n(y) \d y.
\]
   Note that $\lim_{\epsilon\downarrow0} f_\epsilon(x)=1$ and $g_{\epsilon}(y-n)/\sqrt{2\pi}$ is the density of normal random variable with mean $n$ and variance $\epsilon$. The functions $H_n(y)$ are bounded and continuous, thus we can take the limit as $\epsilon \to 0^+$ and obtain \eqref{eq:Laplace_transform_exp_functional}. 
\end{proof}

Corollary \ref{corollary3} could be interesting for applications in the study of exponential functionals of L\'evy processes -- these are random variables of the form 
$$
I_v=\int_0^v e^{\alpha X_u} \d u. 
$$
Assume we managed to compute the right-hand side of 
\eqref{eq:Laplace_transform_exp_functional} in closed form. Then by performing analytic continuation in $\lambda$ and then in $\alpha$ we could obtain an expression for 
$$
\e\Big[  e^{\i \lambda\int_0^v e^{\alpha  X_u
 }\d u} |X_0=0 \Big],
$$
which is the characteristic function of $I_v$. We leave investigation of applications of Corollary \ref{corollary3} to the study of exponential functionals of L\'evy processes for future research.

\subsection{Brownian excursion and meander}\label{subsection_Brownian_excursion}


 We consider the Markov process  $X=\{X_s\}_{0\le s \le 1}$ on $\calX=[0,\infty)$ as the normalized Brownian excursion. This is an inhomogeneous Markov process with $X_0 = 0$ and $X_1 = 0$ fixed, and  with transition probability density function
\[
p_{s_1,s_2}(x_1,x_2)=
\begin{cases}
\sqrt{8\pi} h_{s_2}(x_2) h_{1-s_2}(x_2),& \mbox{ if }  \; 0=s_1<s_2<1, \; x_1=0, \; x_2>0,\\
\\
\displaystyle
\frac{h_{s_2}(x_2)}{h_{s_1}(x_1)}\tilde p_{s_2-s_1}(x_1,x_2), & \mbox{ if } \; 0<s_1<s_2<1, \; x_1,x_2>0,
\end{cases}
\]
where
\[
\tilde p_t(x_1,x_2) =  \frac{1}{\sqrt{2\pi t}}
e^{-\frac{(x_1-x_2)^2}{2t}}-
\frac{1}{\sqrt{2\pi t}} 
e^{-\frac{(x_1+x_2)^2}{2t}}, t>0, x_1,x_2>0,
\]
is the transition (sub-)probability density function of the Brownian motion killed when hitting zero, and
\begin{equation}\label{h_s_Brownian_Excursion}
h_s(x)=\frac{x}{\sqrt{2\pi (1-s)^3}} e^{-
\frac{x^2}{2(1-s)}}, \quad s\in(0,1), x>0.
\end{equation}
Note that the functions $h_s$ play the role of Doob's $h$-transform. 
Furthermore, 
the kernel of killed Brownian motion  $\tilde p_{t}(x_1,x_2)$ has spectral representation
$$
\tilde p_{t}(x_1,x_2)= \frac{2}{\pi} \int_0^{\infty} e^{-\frac12t {y^2}} \sin(x_1 y) \sin( x_2 y) \d y.
$$
This formula can be easily checked by expressing the product of sine functions as a sum of cosine functions and computing the resulting integral, see \cite{Bryc_Wang_2018}. 
 To put the process $X$ into our framework, we consider the Hilbert space $L^2(\calX,\mu) = L^2((0,\infty), \d x)$, take $\phi(y)=y^2/2$ and $\calF$ as the Fourier sine transform, which is initially defined for smooth functions $f: (0,\infty)  \mapsto \c$ with compact support  via the integral representation
$$
{\mathcal F} f(y)=\sqrt{\frac{2}{\pi}} \int_0^{\infty} f(x) \sin(x y) \d x
$$
and then extended to $L^2((0,\infty), \d x)$. 
Note that the Fourier sine transform is an involution on $L^2((0,\infty), \d x)$, that is, ${\mathcal F}^2=I$. 
With this setup, the transition semigroup of the process $X$ can be written as in \eqref{eqn:semigroup_X} when restricted to $\bfS = (0,1)$.

Next, we take $\calY=(0,\infty)$ and a process $Y=\{Y_t\}_{t\ge 0}$ to be the radial part of a three-dimensional Cauchy process, that is 
\begin{equation}\label{eq:Cauchy}
Y_t=\sqrt{Z_{1,t}^2+Z_{2,t}^2+Z_{3,t}^2},
\end{equation}
where $\{Z_{i,t}\}_{t\ge 0}, i=1,2,3$ are independent Cauchy processes. The process $Y$ is  a time-homogeneous Markov process with the transition probability kernel 
\begin{align}
q_t(y_1,y_2)&=\frac{y_2}{y_1} \frac{2}{\pi} \int_0^{\infty} e^{-t  x} \sin(y_1 x) \sin( y_2 x) \d x\nonumber\\
\label{eq:q}
&=\frac{y_2}{y_1} \frac{1}{\pi }
\Big[ \frac{t}{t^2+(y_1-y_2)^2}-\frac{t}{t^2+(y_1+y_2)^2} \Big],
\quad t>0, y_1,y_2>0,
\end{align}
see  \cite[Corollary 1]{kyprianou21doob}.
The process $Y$ can also be identified as the square root of the so-called $1/2$-stable Biane process, which is a homogeneous Markov process that appeared recently in limit theorems concerning Askey--Wilson processes \cite{bryc16local,Bryc_Wang_2018,bryc19limit}. This process also has connections to free probability \cite{biane98processes}.  Taking $j_t(y)=y$ and $\psi(x)=x$ we see that the transition semigroup of the process $Y$ is described as in \eqref{eqn:semigroup_Y}.

We let $\mathbb B^{\rm ex}$ denote a normalized Brownian excursion.

\begin{proposition}\label{prop_B_ex}
Assume that 
$$
0<s_0<s_1< \dots < s_{n}<1, \;\;\;
0=t_0<t_1<\dots<t_{n}. \;\;\; 
$$
Let $f : (0,\infty) \to \c$  be $g : (0,\infty) \to \c$ are functions such that $h_{s_n} f\in L^2((0,\infty), \d x)$ 
and 
\[
y g(y)= [\calF (h_{s_n}f)](y).  
\]
Then 
\begin{equation}\label{eq:duality_B_ex}
\e\Big[ e^{- \sum\limits_{k=0}^{n-1} \Delta t_k{\mathbb B}^{\textnormal{ex}}_{s_k}  }  f({\mathbb B}^{\textnormal{ex}}_{s_n})\Big]=
2\int_0^{\infty}  \e\Big[ e^{-\frac{1}{2}\sum\limits_{k=0}^{n-1} \Delta s_kY_{t_{k+1}}^2 
} g(Y_{t_n})| Y_{0}=y\Big] y^2 e^{-\frac{1}{2} s_0y^2} \d y.
\end{equation}
\end{proposition}
\begin{proof}
To apply Corollary \ref{corollary1}, we take $h_s$ as in \eqref{h_s_Brownian_Excursion}  and $j_t(y) = y$.  We set 
$$
{\ell}^X_{\rm init}(x)=\sqrt{8\pi} h_{s_0}(x) h_{1-s_0}(x), \;\;\;
{\ell}^Y_{\rm init}(y)=2 y^2 e^{-\frac{1}{2} s_0y^2}.
$$
We check readily that 
$\ell_{\rm init}^X/h_{s_0}\in L^2((0,\infty),\d x)$ and, using $\int_0^\infty xe^{-tx^2/2}\sin(xy)\d x  = \pi h_t(y)$, or equivalently $[\calF(x e^{-tx^2/2})](y) = \sqrt{2\pi } h_t(y)$, that 
\begin{equation}\label{eq:l_t_Fs_transform}
\frac{\ell_{\rm init}^Y(y)}{j_{t_0}(y)}=2y e^{-\frac{1}{2} s_0y^2} = \sqrt{8\pi}[\calF(h_{1-s_0})](y) = \Big[\calF\Big(\frac{\ell_{\rm init}^X}{h_{s_0}}\Big)\Big](y).
\end{equation}
Now \eqref{eq:duality_B_ex} follows from  Corollary \ref{corollary1}.
\end{proof}

If we set $f \equiv 1$ in Proposition \ref{prop_B_ex}, then 
$$
g(y) = \frac{1}{y} \calF h_{s_n}(y)=\frac{1}{\sqrt{2\pi}} e^{-\frac{1}{2} (1-s_n) y^2},
$$
and, \eqref{eq:duality_B_ex} becomes
\begin{align*}
\e\Big[ e^{- \sum\limits_{k=0}^{n-1} \Delta t_k{\mathbb B}^{\textnormal{ex}}_{s_k}  }  \Big]&=
{\sqrt {\frac2\pi}}\int_0^{\infty}  \e\Big[ e^{-\frac{1}{2}\sum\limits_{k=0}^{n-1} \Delta s_kY_{t_{k+1}}^2 
}e^{-\frac12(1-s_n)Y_{t_n}^2}| Y_{0}=y\Big] y^2 e^{-\frac{1}{2} s_0y^2} \d y.\end{align*}
After regrouping exponential functions and relabelling (note that $s_n$ is not used in the re-labelling)
\[
\tilde s_k = s_{k-1}, k=1,\dots,n, \tilde s_0 = 0, \tilde s_{n+1} = 1,
\]
 we end up with \eqref{eq:SPL},
which we recall here for convenience:
\[
\e\Big[ e^{- \sum\limits_{k=1}^{n} (t_k-t_{k-1}){\mathbb B}^{\textnormal{ex}}_{\tilde s_k}}  \Big]=
\sqrt{\frac2\pi}\int_0^{\infty}  \e\Big[ e^{-\frac{1}{2}\sum\limits_{k=0}^{n} (\tilde s_{k+1}-\tilde s_k)Y_{t_{k}}^2 
}| Y_{0}=y\Big] y^2  \d y.
\]

Next, we will discuss the example of Brownian meander. The setup and the computations here are very similar to that of Brownian excursion, as both processes are related to the kernel of killed Brownian motion $\tilde p_t$, but up to different $h$-transforms.   

This time, we take a Markov process  $X=\{X_s\}_{0\le s \le 1}$ on $\calX=[0,\infty)$ with $X_0 = 0$ and the transition probability density function
\[
p_{s_1,s_2}(x_1,x_2)=\begin{cases}
\sqrt{8\pi}h_{1-s_2} (x_2)\tilde h_{s_2}(x_2),& \mbox{ if } \; 0=s_1<s_2\le 1, \;x_1 = 0, \;x_2>0,\\\\
\displaystyle \frac{\tilde h_{s_2}(x_2)}{\tilde h_{s_1}(x_1)}\tilde p_{s_2-s_1}(x_1,x_2), & \mbox{ if }\; 0<s_1<s_2\le 1, \; x_1,x_2>0,
\end{cases}
\]
with
$$
\tilde h_s(x)=\frac{1}{\sqrt{2\pi (1-s)}} \int_0^x  e^{-\frac{u^2}{2(1-s)}} \d u = \frac12\int_0^\infty \tilde p_{1-s}(x,y)\d y, \;\;\; s\in(0,1),
$$
and by continuity $\tilde h_1(x) = 1/2$. 
The process $X$ is the Brownian meander, denoted by ${\mathbb B}^{\textnormal{me}}$ below. 
We take the process $Y=\{Y_s\}_{s\ge 0}$ the same as introduced earlier in \eqref{eq:Cauchy}.

\begin{proposition}\label{prop_B_me}
Assume that 
$$
0<s_0<s_1<\dots<s_{n}<1, \;\;\; 0=t_0<t_1< \dots < t_{n}. 
$$
Let $f : (0,\infty) \to \c$  and $g : (0,\infty) \to \c$ be functions such that $\tilde h_{s_n} f \in L^2((0,\infty), \d x)$ 
and 
\[
y g(y)= [\calF(\tilde h_{s_n}f)](y).  
\]
Then 
\begin{equation}\label{eq:duality_B_me}
\e\Big[ e^{- \sum\limits_{k=0}^{n-1} \Delta t_k{\mathbb B}^{\textnormal{me}}_{s_k} }  f({\mathbb B}^{\textnormal{me}}_{s_n})\Big]=
2\int_0^{\infty}  \e\Big[ e^{-\frac{1}{2}\sum\limits_{k=0}^{n-1} \Delta s_kY_{t_{k+1}}^2 
} g(Y_{t_n})| Y_{0}=y\Big] y^2 e^{-\frac{1}{2} s_0y^2 } \d y.
\end{equation}
\end{proposition}
\begin{proof}
This time we take $h_s = \tilde h_s, j_t(y) = y$, $\ell_{\rm init}^X(x) = \sqrt{8\pi}h_{1-s_0}(x)\tilde h_{s_0}(x)$ and $\ell_{\rm init}^Y(y) = 2y^2 e^{-s_0y^2/2}$. The identity \eqref{eq:l_t_Fs_transform} remains to hold and \eqref{eq:duality_B_me} follows from Corollary \ref{corollary1}.
\end{proof}
Now we derive the following dual Laplace transform for Brownian meander.
Consider $0=\tilde s_0 <\tilde s_1<\cdots<\tilde s_{n+1} = 1$, $c_1,\dots,c_n>0, c_{n+1}\ge 0$, and $t_0 = 0, t_k = c_1+\cdots+c_k, k=1,\dots,n$. Then, 
\begin{equation}\label{eq:duality_Bme}
\e\Big[e^{-\sum\limits_{k=1}^{n+1}c_k\mathbb B^{\textnormal{me}}_{\tilde s_k}}\Big] = \sqrt{\frac2\pi}\int_0^\infty \e\Big[e^{-\frac12\sum\limits_{k=0}^{n}\Delta \tilde s_kY^2_{t_k}}\frac1{c_{n+1}^2+Y^2_{t_{n}}} | Y_{0} = y\Big]y^2\d y.
\end{equation}
The above identity has been obtained in \cite[Proposition 4.3 and (4.5)]{bryc23asymmetric} (therein expressed in terms of $\zeta_t = Y_t^2$). 
Note that $\mathbb B_1^{\textnormal {me}}$ is involved on the left-hand side, and  the time indices of $Y$ process involved on the right-hand side, $t_0,\dots,t_n$, do not depend on $c_{n+1}$.

We derive \eqref{eq:duality_Bme} from Proposition \ref{prop_B_me}. Note that $s_n<1$ is required therein, and we need a small continuity argument to go around it. Take $f(x) = e^{-\Delta t_nx}$, which is a continuous function. Therefore, by continuity now \eqref{eq:duality_B_me} holds with $s_n = 1$  and the relation
$yg(y) = [\calF (\tilde h_{s_n}f)](y)$ replaced by $yg(y) = [\calF(f/2)](y)$.
We apply this modified identity. By Laplace transform of the sine function, we have
\[
g(y) = \frac1{y}\calF f(y) = \frac1{2y}\sqrt{\frac 2\pi}\int_0^\infty\sin(xy) e^{-\Delta t_n x}\d x = \frac1{\sqrt{2\pi}}\frac y{(\Delta t_n)^2+y^2},
\]
In this way, we obtain \eqref{eq:duality_Bme}.
\begin{remark}
When the duality representation concerning Brownian meander first appeared in \cite[Theorem 1.2]{Bryc_Wang_2018},  the formula therein, with $c_{n+1} =0$ (recall $\tilde s_n<1$), is in fact one more step from \eqref{eq:duality_Bme} via `reversing' the Markov process $Y$ by conditioning on the value of $Y_{t_n}$ instead. Note that $Y$ is reversible: $ x^2 q_t(x,y) =  y^2 q_t(y,x)$ (recall $q_t$ in \eqref{eq:q}). So now \eqref{eq:duality_Bme} becomes
\begin{equation}\label{eq:SPL1}
\e\Big[e^{-\sum\limits_{k=1}^{n}c_k\mathbb B^{\textnormal{me}}_{\tilde s_k}}\Big] = \sqrt{\frac2\pi}\int_0^\infty \e\Big[e^{-\frac12\sum\limits_{k=0}^{n}\Delta\tilde s_kY^2_{t_k}}| Y_{t_{n}} = y\Big]\d y.
\end{equation}
After relabeling, we recognize the above as the one in \cite[Theorem 1.2]{Bryc_Wang_2018} (which was stated in terms of process $\zeta_t = Y_t^2$). For comparison, there it was also established that
\[
\e\Big[e^{-\sum\limits_{k=1}^{n}c_k\mathbb B^{\textnormal{ex}}_{\tilde s_k}}\Big] = \sqrt{\frac2\pi}\int_0^\infty \e\Big[e^{-\frac12\sum\limits_{k=0}^{n}\Delta\tilde s_kY^2_{t_k}}| Y_{t_{n}} = y\Big]y^2\d y.
\]
The proof establishing \eqref{eq:SPL1} in  
\cite{Bryc_Wang_2018} was different from the one presented here: it was by working with a Brownian excursion with a randomized end point.
\end{remark}



\subsection{CIR diffusion and a birth-and-death process}\label{subsection_CIR}


Let $X=\{X_s\}_{s\ge 0}$ be the square-root diffusion process, also known in the finance literature as the Cox--Ingersoll--Ross process, see \cite{LINETSKY2007223}. This process satisfies the SDE
$$
\d X_s=(\alpha+1-X_s) \d s+ \sqrt{2 X_s} \d W_s,
$$
where $\{W_s\}_{s\ge 0}$ is a standard Brownian motion.
We assume that $\alpha \ge 0$, so that zero is an inaccessible (entrance) boundary for $X$. The density of the stationary measure is known to be 
\[
\mu_\alpha(x)=\frac{x^{\alpha}}{\Gamma(\alpha+1)} e^{-x}, \;\;\; x>0. 
\]
Let $L_k^{(\alpha)}(x)$ be Laguerre polynomials, see \cite[Section 8.97]{gradshteyn07table}. When $\alpha=0$ we write simply $L_k^{(0)}(x)=L_k(x)$. These polynomials satisfy the following properties:
\begin{itemize}
  \item[(i)]  the orthogonality relation
\[
\int_0^{\infty} L_{k_1}^{(\alpha)}(x) L_{k_2}^{(\alpha)}(x) \mu_\alpha(x) \d x=\frac{\delta_{k_1,k_2}}{\pi^{(\alpha)}_{k_2}},  \quad k_1,k_2\ge 0,
\]
where 
$$
\pi^{(\alpha)}_k:=\frac{k!}{(1+\alpha)_k}, \;\;\; k\ge 0, 
$$
\item[(ii)]  the three-term recurrence relation 
$$
(k+1) L_{k+1}^{(\alpha)}(x)-(2k+\alpha+1) L_k^{(\alpha)}(x) + (k+\alpha) L_{k-1}^{(\alpha)}(x) = - x L_k^{(\alpha)}(x),  \quad k\ge 0,
$$
with $L_0^{(\alpha)}(x) = 1, L_{-1}^{(\alpha)}(x) = 0$. 
\item[(iii)]  and $L_k^{(\alpha)}(x)$ is a solution to the ODE
$$
(\alpha+1-x) f'(x)+x f''(x)+k f(x)=0,
$$
for each $k\ge 0$. 
\end{itemize}
The transition density of the process $X$ is known to be \cite{LINETSKY2007223}
\begin{align}\label{Hille_Hardy}
p_s(x_1,x_2)&=c_s \big( e^{s} x_2 /x_1\big)^{\alpha/2}
\exp\big(-c_s (x_1 e^{-s}+x_2) \big) I_{\alpha}\big(2 c_s \sqrt{x_1 x_2 e^{-s}} \big)\\ \nonumber
&=
\mu_\alpha(x_2)\sum\limits_{k\ge 0} e^{-k s} L_k^{(\alpha)}(x_1)L_k^{(\alpha)}(x_2)  \pi^{(\alpha)}_k, \;\;\; t>0,\; x_1>0, \; x_2>0.  
\end{align}
Here we denoted $c_s:=1/(1-e^{-s})$ and the second equality in the above formula follows from Hille--Hardy formula, see equation (8.976.1) in \cite{gradshteyn07table}. 
The infinite series in \eqref{Hille_Hardy} converges pointwise. This can be verified using the asymptotic expansion of Laguerre polynomials given in the formula (8.978.3) in 
\cite{gradshteyn07table}.

Consider now a birth and death process $Y=\{Y_t\}_{t\ge 0}$ that is supported on $\mathbb Z^+:=\{0,1,2,\dots\}$ and is specified via the transition probabilities $
q_t(k_1,k_2)=\p(Y_t=k_2 | Y_0=k_1)$ satisfying, for all $k_1, k_2 \ge 0$ as $t\to 0^+$
\begin{align*}
q_t(k_1,k_2)=
\begin{cases}
(k_1+\alpha)t+o(t), \;\;\; &
{\textnormal{ if }} k_2=k_1-1,\\
1-(2k_1+\alpha+1)t+o(t), \;\;\; &
{\textnormal{ if }} k_2=k_1,\\
(k_1+1)t+o(t), \;\;\; &
{\textnormal{ if }} k_2=k_1+1, \\
o(t), \;\;\; &
{\textnormal{ if }} |k_2-k_1|>1.\\
\end{cases}
\end{align*}
According to Karlin and McGregor \cite{karlin1957}, the transition probabilities for process $Y$ can be expressed as follows:  
$$
q_t(k_1,k_2)=\pi^{(\alpha)}_{k_2} \int_0^{\infty} e^{-t x}  L_{k_1}^{(\alpha)}(x)  L_{k_2}^{(\alpha)}(x) \mu_\alpha(x) \d x, \;\;\; k_1\ge 0, k_2\ge 0. 
$$

Let us now reinterpret the above discussion in the Hilbert space framework of Section \ref{section_general_case}. We consider spaces 
$L^2(\calX,\mu) = L^2((0,\infty), \mu_\alpha(x) \d x)$ and $L^2(\calY, \nu_\alpha) = L^2({\mathbb Z}^+, \nu_\alpha)$, 
where $\nu_\alpha$ is the discrete measure supported on ${\mathbb Z}^+$ having weights 
\[
\nu_\alpha(\{k\})=\pi^{(\alpha)}_k, \quad k\in\mathbb Z^+.
\] Define now a transformation ${\mathcal F}_\alpha$, which takes a function $f \in L^2((0,\infty), \mu_\alpha(x) \d x)$ and returns a sequence $\{{\mathcal F}_\alpha f(k)\}_{k\ge 0}$ (viewed as an element in $L^2(\mathbb Z^+,\nu)$) of coefficients of expansion of $f$ in Laguerre polynomials: 
$$
{\mathcal F}_\alpha f(k)= \int_0^{\infty} f(x) L_k^{(\alpha)}(x) \mu_\alpha(x) \d x, \quad k\in\mathbb Z^+.
$$
Orthogonality and completeness of Laguerre polynomials in $L^2((0,\infty), \mu_\alpha(x) \d x)$ imply that ${\mathcal F}_\alpha$ is a unitary operator from 
$L^2((0,\infty), \mu_\alpha(x) \d x)$ onto $L^2({\mathbb Z}^+, \nu_\alpha)$, and the inverse transformation is given by 
$$
{\mathcal F}^{-1}_\alpha g(x)=\sum\limits_{k\ge 0} g(k) L_k^{(\alpha)}(x) \pi^{(\alpha)}_k, \quad\{g(k)\}_{k\in \mathbb Z^+} \in L^2(\mathbb Z^+,\nu_\alpha).
 $$
With these definitions in place, we can now describe the semigroups of process $X$ and $Y$ as follows: 
\begin{align*}
   {\mathcal P}_{s_1,s_2} f(x)&= {\mathbb E}[f(X_{s_2}) \vert X_{s_1}=x]=[{\mathcal F}_\alpha^{-1} e^{(s_1-s_2) y} {\mathcal F}_\alpha f](x), \;\;\; s_1<s_2,\\
   {\mathcal Q}_{t_1,t_2}g(y)&={\mathbb E}[g(Y_{t_2}) \vert Y_{t_1}=y]=[{\mathcal F}_\alpha e^{(t_1-t_2) x} {\mathcal F}_\alpha^{-1} g](y),\;\;\; t_1<t_2,
\end{align*}
where $f\in L^2((0,\infty), \mu_\alpha(x) \d x)$ and $g\in L^2({\mathbb Z}^+, \nu_\alpha)$. As before, the terms $e^{(t_1-t_2) x}$ and $e^{(s_1-s_2) y}$ stand for  multiplication operators: the first multiplies a function $f(x)$ by $e^{(t_1-t_2) x}$ and the second sends a sequence  $g=\{g(k)\}_{k \ge 0}$ into a sequence
$\{e^{(s_1-s_2)k} g(k)\}_{k \ge 0}$. 

Now we can apply Theorem \ref{thm0} and obtain the following result. 

\begin{proposition}
Assume that 
$$
0=s_0<s_1<\dots<s_{n}, \;\;\; 0=t_0<t_1< \dots < t_{n}. 
$$
Assume that $g=\{g(k)\}_{k\ge 0}$ is a sequence in $L^2({\mathbb Z}^+, \nu_\alpha)$ and let
$$
f(x)={\mathcal F}^{-1}_\alpha g(x)=\sum\limits_{k\ge 0} g(k) L_k^{(\alpha)}(x) \pi^{(\alpha)}_k.
$$
Denote
\[
{\mathsf F}(x):=\e\Big[ e^{-\sum\limits_{j=0}^{n-1}  \Delta t_jX_{s_j}
}
f(X_{s_{n}}) | X_{0}=x \Big], \;\;\; x>0,
\]
and 
\[
{\mathsf G}(k):=\e\Big[ e^{-\sum\limits_{j=0}^{n-1}  \Delta s_jY_{t_{j+1}}}g(Y_{t_{n}}) | Y_{0}=k \Big], 
\;\;\; k\in\mathbb Z^+. 
\]
Then ${\mathsf F} \in L^2((0,\infty), \mu_\alpha(x)\d x)$ and 
${\mathsf G}={\mathcal F}_\alpha {\mathsf F}$.  
\end{proposition}
\begin{proof}
The proof follows by applying Theorem \ref{thm0}, we take $h_s = 1$ and $j_t = 1$.
\end{proof}

For example, if we take $g(k)=\delta_{k,m}$, then $f(x)=L_m^{(\alpha)}(x) \pi^{(\alpha)}_m$, $\mathsf F = \calF^{-1}\mathsf G$ reads as
\[
    \pi_m^{(\alpha)} \e\Big[ e^{-\sum\limits_{j=0}^{n-1}  \Delta t_jX_{s_j}
}
L_m^{(\alpha)}(X_{s_{n}}) | X_{0}=x \Big]= 
\sum\limits_{k\ge 0}\e\Big[ e^{-\sum\limits_{j=0}^{n-1}  \Delta s_jY_{t_{j+1}}} {\mathbf 1}_{\{Y_{t_{n}}=m\}} | Y_{0}=k \Big] L_{k}(x) \pi_{k}.
\]

Let us now consider another example. We take a CIR process $X$ and a birth and death process $Y$ with $\alpha=0$, so that $\pi_k \equiv \pi_k^{(0)}=1$ for all $k\ge 0$ and $\mu_0(x)=\exp(-x)$. We set $g(k)=\delta_{k,0}$ so that $f(x)=
L_0(x) \pi_0=1$ for all $x>0$. Using the series expansion (formula (8.975.1) in \cite{gradshteyn07table})
$$
 \exp\Big(-x \frac{1-\lambda}{\lambda}\Big) 
= \lambda \sum\limits_{k\ge 0} L_k(x) (1-\lambda)^k = \calF_0^{-1}g(x), \;\;\; 0<\lambda<1,
$$
with $g(k) = \lambda (1-\lambda)^k$,
and by Corollary \ref{corollary1} we obtain the following result: 
\begin{align}\label{eqn:CIR_exponential}
\int_0^{\infty} \e\Big[ e^{-\sum\limits_{j=0}^{n-1}  \Delta t_jX_{s_j}
} | X_{0}=x \Big]  e^{-\frac{x}{\lambda}} \d x 
= \sum\limits_{k \ge 0} \e\Big[ e^{-\sum\limits_{j=0}^{n-1}  \Delta s_jY_{t_{j+1}}} {\mathbf 1}_{\{Y_{t_{n}}=0\}} | Y_{0}=k \Big]  \lambda (1-\lambda)^{k}.
\end{align}
We recognize the left-hand side (up to a multiple of $1/\lambda$) as a joint Laplace transform of  $(X_{s_0},\dots,X_{s_{n-1}})$, where process $X$ starts at the $X_0 \sim {\textnormal{Exp}}(1/\lambda)$ (an exponential distribution with parameter $1/\lambda$), and the right-hand side of \eqref{eqn:CIR_exponential} gives us the joint Laplace transform of $(Y_{t_1},\dots,Y_{t_n})$ on the event $Y_{t_n}=0$, where the birth-and-death process $Y$ starts at $Y_0 \sim {\textnormal{Geo}}(\lambda)$ (a geometric distribution with parameter $\lambda$). 


\section{Brownian motion killed at an exponential rate and continuous dual Hahn process}\label{section_kBM_and_dual_Hahn}

In this section we consider an example that first appeared in 
\cite{Bryc_et_al_2023}. This example does not fit into the framework presented in Section \ref{section_general_case}, as we can not work with Hilbert spaces. Instead, we have to work with certain Banach spaces, which requires relaxing the conditions in Assumption \ref{assump:1}. Thus, our first goal is to present a more general setup that will allow us to describe the duality between  a Brownian motion killed at an exponential rate and the so-called continuous dual Hahn process, recently introduced in \cite{corwin24stationary} (see also \cite{BRYC2022185}). 
\subsection{A more general framework}
As before, we start with the two Markov processes $X=\{X_s\}_{s\in \bfS}$  and 
 $Y=\{Y_t\}_{t \in {\bfT}}$ taking values in $\calX \subseteq \r$ and $\calY \subseteq \r$ respectively. However, in order to describe their Markov semigroups, we would require the following.
\begin{assumption}\label{assump:2}
There exist 
 
 \begin{itemize}
 \item[(i)]
 two Banach spaces ${\mfA}$ and ${\mfB}$ of measurable complex-valued functions on $\calX$ and $\calY$ respectively;
 \item[(ii)] two continuous functions $\phi: \calY \to \c$ and $\psi: \calX \to \c$, and operators   $\calF$ and $\calG$ such that for every $\delta>0$
 \begin{align}\label{condition_calF}
     &e^{-\delta \phi} \calF f \in \mfB, \;\;\;
     {\textnormal{ for every}} \;\; f\in \mfA,  \\
     \label{condition_calG}
     &e^{-\delta \psi} \calG g \in \mfA, \;\;\;
     {\textnormal{ for every}} \;\; g\in \mfB, 
 \end{align}
 \item[(iii)]  measurable functions $h_s(x): s \in \bfS, x\in \calX\to{(0,\infty)}$ and $j_t(y): t \in \bfT, y\in \calY\to{(0,\infty)}$, 
\end{itemize}
such that Markov semigroups of the process $X$ and $Y$ are given by 
 \begin{align}
 \label{eqn:semigroup_X2}
{\mathcal P}_{s_1,s_2} f(x)&=\left[\frac{1}{h_{s_1}} \calG e^{(s_1-s_2) \phi} \calF (h_{s_2} f)\right](x), \;\;\; 
s_i \in {\bfS}, \; s_1<s_2, \\
\label{eqn:semigroup_Y2}
{\mathcal Q}_{t_1,t_2}g(y)&=\left[\frac{1}{j_{t_1}} \calF e^{(t_1-t_2) \psi} \calG (j_{t_2} g)\right](y), \;\;\; 
t_i \in {\bfT}, \; t_1<t_2. 
\end{align} 
The identities \eqref{eqn:semigroup_X2} and \eqref{eqn:semigroup_Y2} are assumed to hold for functions $f(x)$ and $g(y)$ for which $h_{s_2} f \in \mfA$ and $j_{t_2} g \in \mfB$. 
\end{assumption}

Now we are able to state and prove the following analogue of Theorem \ref{thm0}.
\begin{theorem}\label{thm3}
Let conditions in Assumption \ref{assump:2} be true. Let $s_i \in \bfS$ and $t_i \in \bfT$
be such that 
$$
s_0<s_1<\dots<s_{n}, \;\;\; t_0<t_1< \dots < t_{n}. 
$$
Assume that $f : \calX \to \r$ and $g : \calY \to \r$ are functions such that $h_{s_n} f \in \mfA$ and 
$j_{t_{n}} g={\mathcal F}(h_{s_{n}}f)$.
Denote
\[
{\mathsf F}(x):=\e\Big[ e^{-\sum\limits_{k=0}^{n-1}  \Delta t_k\psi(X_{s_k})}
f(X_{s_{n}}) | X_{s_0}=x \Big], \;\;\; x \in \calX,
\]
and 
\[
{\mathsf G}(y):=\e\Big[ e^{-\sum\limits_{k=0}^{n-1}  \Delta s_k\phi(Y_{t_{k+1}})}g(Y_{t_{n}}) | Y_{t_0}=y \Big], 
\;\;\; y\in \calY. 
\]
Then $h_{s_0} {\mathsf F} \in \mfB$ and the functions ${\mathsf F}$ and ${\mathsf G}$ satisfy
$j_{t_0} {\mathsf G}={\mathcal F} (h_{s_0}{\mathsf F}).$
\end{theorem}
\begin{proof}
The proof is virtually identical to that of Theorem \ref{thm0}. We write down equations 
\eqref{eq:thm0_proof1} and 
\eqref{eq:thm0_proof2}, but with $\calF^{-1}$ replaced by $\calG$. The argument that demonstrates that $h_{s_0} F \in \mfA$ is the same, except that we use \eqref{condition_calF}
and \eqref{condition_calG} so justify that the relevant functions belong to $\mfB$ or $\mfA$. 
\end{proof}

\subsection{An example}
We take $X=\{X_s\}_{s\geq 0}$ as a stochastic process defined by the Markov generator
\begin{equation*}
{\mathcal L}^{X}f(x)=f''(x)-e^{2x} f(x).
\end{equation*}
This process can be identified with a Brownian motion, scaled by $\sqrt{2}$, and killed at rate $e^{2x}$.
It is known  \cite[Example 3.7]{sousa18spectral} that the transition probability density function of the process $X$ is given by 
\begin{equation}\label{eq:Y}
p_s(x_1,x_2)=\int_0^{\infty} 
e^{-sy^2} K_{\i y}(e^{x_1}) K_{iy}(e^{x_2}) \nu(\d y), s>0
\end{equation}
where 
$$
\nu(\d y):=\frac{2}{\pi^2} y\sinh(\pi y)\, \d y,
$$
and $K_{\i y}(z)$ is the modified Bessel function of the second kind, defined as 
\[
K_{\i y}(z)=\int_0^{\infty} e^{-z \cosh(u)} \cos(y u) \d u, \;\;\; z>0.  
\]
The kernel $p$ introduced above in \eqref{eq:Y} is known as the Yakubovich heat kernel,  see \cite{sousa18spectral} and \cite[Section  3.2]{Bryc_et_al_2023}. 
In particular, the Markov semigroup of $X$ can be described in terms of the Kontorovich-Lebedev transform, defined as 
\[
{\mathcal F}_{{\textnormal{KL}}}f (y)=\int_{\r} f(x) K_{\i y}(e^x) \d x, \;\;\; y>0.
\]
It is known (see Section 2.3 in \cite{Yakubovich_1996})
that $\calF_{\textnormal{KL}}$ can be extended to a unitary operator 
\begin{equation}
\calF_{\textnormal{KL}}:L^2(\r,\d x)\to L^2((0,\infty),\nu(\d y)).
\label{mu}
\end{equation}
The inverse operator is the extension to 
$L^2((0,\infty),\nu(\d y))$ 
of the integral operator
\[
{\mathcal F}_{{\textnormal{KL}}}^{-1} g(y)= \int_0^{\infty} K_{\i y}(e^{x}) g(y) \nu(\d y).
\]

The Markov semigroup of the process $X$ is given by
\begin{equation}\label{eq:P_KL}
   {\mathcal P}_{s_1,s_2} f(x)= {\mathbb E} [ f(X_{s_2}) | X_{s_1}=x]= 
   \Big[{\mathcal F}_{{\textnormal{KL}}}^{-1} e^{(s_1-s_2) y^2}
    {\mathcal F}_{{\textnormal{KL}}} f\Big](x), \;\;\;
    0 \le s_1<s_2,
\end{equation}
for any bounded $f\in L^2(\r, \d x)$. 

Next, we fix a parameter $\mathfrak b>0$, and consider a process $Y=\{Y_t\}_{0\le t \le \mathfrak b}$ defined by transition probability density function 
\begin{align*}\nonumber
q^{(\mathfrak b)}_{t_1,t_2}(y_1,y_2)\d y&=\frac{1}{8 \Gamma(t_2-t_1)}\left| \Gamma\left(\frac{t_2-t_1+\i (y_1+y_2)}2\right) \Gamma\left(\frac{t_2-t_1+\i (y_1-y_2))}2\right) \right| ^2\\
\nonumber
&\qquad \qquad \times 
\frac{\vert \Gamma(\frac{1}{2}(\mathfrak b-t_2+\i y_2)) \vert^2}
{\vert \Gamma(\frac{1}{2}(\mathfrak b-t_1+\i y_1)) \vert^2} 
\nu(\d y_2),
\end{align*}
where $y_i>0$ and $0\le t_1<t_2\le \mathfrak  b$. The fact that $q^{(\mathfrak b)}_{s_1,s_2}(y_1,y_2)$ is a transition probability density function of a Markov process was established in \cite[Appendix B]{Bryc_et_al_2023}. 
In what follows, we will need two integral identities: 
\begin{align}
    \int_{\mathbb R}e^{tx} K_{\i y_1}(e^x) \d x & =2^{t-2} \Big \vert 
    \Gamma\Big(\frac{t+\i y_1}{2}\Big) \Big\vert^2, \label{eq:Y1}\\
    \int_{\r } e^{t x} K_{\i y_1}(e^x) K_{\i y_2} (e^x) \d x &=\frac{2^{t-3}}{\Gamma(t)}
\left| \Gamma\left(\frac{t+\i (y_1+y_2)}2\right) \Gamma\left(\frac{t+\i (y_1-y_2)}2\right)\right| ^2,\nonumber
 \end{align}
 where $t>0$, $y_1>0$ and $y_2>0$. These two identities can be found in 
   \cite[6.8 (26)]{erdelyi1954fg} and  \cite[6.8 (48)]{erdelyi1954fg}. 
With these identities, we can write {\em formally} the Markov semigroup of the process $Y$ as 
\begin{equation}\label{eq_Q_dual_Hahn}
    \calQ^{(\mathfrak b)}_{t_1,t_2} g(y)={\mathbb E} [ g(Y_{t_2}) | Y_{t_1}=y]= 
    \Big[\frac{1}{j^{(\mathfrak b)}_{t_1}} {\mathcal F}_{{\textnormal{KL}}}
e^{(t_1-t_2)\psi} {\mathcal F}_{{\textnormal{KL}}}^{-1}( j^{(\mathfrak b)}_{t_2} g )\Big](y),
\;\;\; 0\le t_1 < t_2\le \mathfrak b,
\end{equation}
where $\psi(x)=-x$ and
\[\label{eq:j}
j^{(\mathfrak b)}_t(y)=2^{\mathfrak b-t-2} \Big\vert \Gamma\Big(\frac{1}{2}(\mathfrak b-t+\i y)\Big) \Big\vert^2, \;\;\; y>0, \; 0\le t\le\mathfrak b.  
\]

The squared process $\{Y_t^2\}_{0\le t \le\mathfrak  b}$ (with appropriately specified law for $Y_0$) is called {the continuous dual Hahn process}, playing an important role in recent studies of stationary measures
for the open KPZ equation (see \cite{corwin24stationary}, also  \cite{BRYC2022185,Bryc_et_al_2023}).

We emphasize that although formally this example (in particular \eqref{eq:P_KL} and \eqref{eq_Q_dual_Hahn}) looks similar to those we have discussed in previous sections (with $j^{(\mathfrak b)}_t$ as above and $h_s = 1$), here certain integrability is lost and hence one cannot apply Theorem \ref{thm0}.  Namely, the problem with \eqref{eq_Q_dual_Hahn} is that the function $e^{(t_1-t_2)\psi(x)}$ is unbounded and  it is  unclear whether $e^{(t_1-t_2)\psi} {\mathcal F}_{{\textnormal{KL}}}^{-1} (j^{(\mathfrak b)}_{t_2} g)
\in L^2(\r, \d x)$ if $j^{(\mathfrak b)}_{t_2} g \in L^2((0,\infty), \nu(\d y))$. This problem was overcome in \cite{Bryc_et_al_2023} by working with $\calF_{\textnormal{KL}}$ and its inverse on certain $L^1$ spaces in place of those $L^2$ spaces indicated in \eqref{mu}.
More precisely, we need the following lemma to  check that the processes $X$ and $Y$ satisfy  Assumption \ref{assump:2}. 
\begin{lemma}[Lemma A.1 in \cite{Bryc_et_al_2023}]\label{L0} Let $\delta>0$.
\begin{itemize}
\item[(i)]  
If $f\in L^1(\r, K_0(e^x) \d x)$   then  $e^{-\delta y^2} {\mathcal F}_{{\textnormal{KL}}} f\in L^1((0,\infty),\nu(\d y))$.
\item[(ii)] If $g\in L^1((0,\infty),\nu(\d y))$  then 
$
e^{\delta x} {\mathcal F}_{{\textnormal{KL}}}^{-1} g
\in L^1(\r, K_0(e^x) \d x)$.
\end{itemize}
\end{lemma}

Now we can show that processes $X$ and $Y$ satisfy Assumption \ref{assump:2}. We take 
 
 \[
 {\mathfrak A}=L^1(\r, K_0(e^x) \d x) \quad \mbox{ and } \quad{\mathfrak B}=L^1((0,\infty),\nu(\d y)),
 \]   
 define $\phi(y)=y^2$ and
$\psi(x)=-x$. It is known that the operator $\calF_{\textnormal{KL}}$ (resp.~$\calF_{\textnormal{KL}}^{-1}$) can be extended to $\mathfrak A$ (resp.~$\mathfrak B$), and we let $\calF$ (resp.~$\calG$) denote this extension.  (It is not clear if the invertibility relation still holds, but this is irrelevant to the discussions here.) 
All the conditions of Assumption \ref{assump:2} are satisfied due to Lemma \ref{L0} and identities
\eqref{eq:P_KL} and  \eqref{eq_Q_dual_Hahn}. 
We have the following.
\begin{proposition}
Assume that 
\[
 0=s_0<s_1< \dots < s_{n}, \;\;\; 0=t_0<t_1<\dots<t_{n}< \mathfrak b.
\]
Denote 
\[
{\mathsf F}(x):=\e\Big[ e^{\sum\limits_{k=0}^{n-1}  \Delta t_kX_{s_k}
+ (\mathfrak b-t_n)X_{s_n}} | X_{0}=x \Big], \;\;\; x \in {\mathbb R},
\]
and 
\[
{\mathsf G}(y):=\e\Big[ e^{-\sum\limits_{k=0}^{n-1}   \Delta s_kY_{t_{k+1}}^2} | Y_{0}=y \Big], 
\;\;\; y>0. 
\]
Then ${\mathsf F} \in L^1(\r, K_0(e^x) \d x)$ and
$j_{0} {\mathsf G} =  {\mathcal F}_{{\textnormal{KL}}}{\mathsf F} $.
\end{proposition}
\begin{proof}
We apply Theorem 
\ref{thm3} with $h_s = 1$, $j_t$ as in \eqref{eq:j}, 
\[
f(x)=e^{(\mathfrak b-t_n)x},
\] and $g(y)=1$ (recall \eqref{eq:Y1}). It suffices to check $f\in\mathfrak A$.  It is known (e.g.~\cite[Eq.(A4) and (A5)]{Bryc_et_al_2023}) that $K_0(e^x)=O(\exp(-e^x))$ 
as $x \to +\infty$ and $K_0(e^x)=O( e^{-\epsilon x})$ as $x\to -\infty$ (for any $\epsilon>0$).  These bounds guarantee that $f \in {\mathfrak A}$. 
\end{proof}

\begin{remark}\label{rem:KPZ}
There are quite a few places where this example looks quite different from earlier ones. First, the most remarkable one is probably the fact that the function $e^{(t_1-t_2)\psi}$ in \eqref{eq_Q_dual_Hahn} is unbounded, as discussed earlier. Second,  the process $Y$ is not time-homogeneous, and (its law and the range of its time index) depending the parameter $\mathfrak b$. Next, the choice of function $f$ depends also on $t_n$ and $\mathfrak b$. 

As mentioned already, the duality formula plays a crucial role in understanding the joint Laplace transform of the so-called stationary measure of open KPZ equation in \cite{Bryc_et_al_2023}. Therein, a corresponding development of Corollary \ref{corollary1} here, and some delicate choices of $\ell_{\rm init}^X$ and $\ell_{\rm end}^X$, are needed. Also needed is a general version of Parseval's identity in the form of
\[
\int_\r f(x) \calG g(x)\d x = \int_0^\infty g(y)\calF f(y)\mu(\d y), f\in \mathfrak A, g\in\mathfrak B.
\]
See \cite[Lemma A.1]{Bryc_et_al_2023}. We omit the details.
\end{remark}
\subsection*{Acknowledgement}
AK's research was partially supported by the Natural Sciences and Engineering Research Council of Canada.
YW's research was partially supported by Army Research Office, US (W911NF-20-1-0139). The authors would like to thank W\l odek Bryc,  Pierre Patie and Mladen Savov for very helpful discussions. The authors are also grateful to an anonymous referee for carefully reading the paper and for providing helpful comments and suggestions.



\def\cprime{$'$} \def\polhk#1{\setbox0=\hbox{#1}{\ooalign{\hidewidth
  \lower1.5ex\hbox{`}\hidewidth\crcr\unhbox0}}}
  \def\polhk#1{\setbox0=\hbox{#1}{\ooalign{\hidewidth
  \lower1.5ex\hbox{`}\hidewidth\crcr\unhbox0}}}


\end{document}